\documentclass[11pt]{article}
\usepackage{latexsym,amsmath,amsthm,amssymb,color}
\usepackage{mathrsfs,wasysym,graphicx,lscape}

\begin{document}
\newtheorem{theorem}{\indent Theorem}[section]
\newtheorem{proposition}[theorem]{\indent Proposition}
\newtheorem{definition}[theorem]{\indent Definition}
\newtheorem{lemma}[theorem]{\indent Lemma}
\newtheorem{remark}[theorem]{\indent Remark}
\newtheorem{corollary}[theorem]{\indent Corollary}

%%%%%%%%%%
\begin{center}
    {\large \bf  On the blow-up solutions for the nonlinear Schr\"{o}dinger equation with combined power-type nonlinearities}
\vspace{0.5cm}\\{ Binhua Feng}\\
{\small Department of Mathematics, Northwest Normal University, Lanzhou, 730070, P.R. China }\\
\end{center}

\renewcommand{\theequation}{\arabic{section}.\arabic{equation}}
\numberwithin{equation}{section}
\footnote[0]{\hspace*{-7.4mm}
%%%%%%%%%%
E-mail: binhuaf@163.com(Binhua Feng)\\
This work is supported by NSFC Grants (No. 11601435, No. 11401478), Gansu Provincial Natural
Science Foundation (1606RJZA010) and NWNU-LKQN-14-6.}

\renewcommand{\baselinestretch}{1.7}
\large\normalsize
\begin{abstract}
This paper is devoted to the analysis of blow-up solutions for the nonlinear Schr\"{o}dinger equation with combined power-type nonlinearities
\[
iu_{t}+\Delta u=\lambda_1|u|^{p_1}u+\lambda_2|u|^{p_2}u.
\]
When $p_1=\frac{4}{N}$ and $0<p_2<\frac{4}{N}$, we prove the existence of blow-up solutions and find the sharp threshold mass of blow-up and global existence for this equation.
This is a complement to the result of Tao et al. (Comm. Partial Differential Equations 32: 1281-1343, 2007). Moreover, we investigate the dynamical properties of blow-up solutions, including $L^2$-concentration, blow-up rates and limiting profile. When $\frac{4}{N}<p_1<\frac{4}{N-2}$($4<p_1<\infty$ if $N=1$, $2<p_1<\infty$ if $N=2$), 
we prove that the blow-up solution with
bounded $\dot{H}^{s_c}$-norm must concentrate at least a fixed amount of the $\dot{H}^{s_c}$-norm
and, also, its $L^{p_c}$-norm must concentrate at least a fixed $L^{p_c}$-norm.

%under the assumption that $\dot{H}^{s_c}$-norm of the blow-up solution is bounded, we prove that the $\dot{H}^{s_c}$-norm of the blow-up solution concentrates at some point and its $L^{p_c}$-norm concentrates as well.

%2010 Mathematics Subject Classification: 35Q55, 35A15.
\end{abstract}
\section{Introduction}
Because of important applications in physics, nonlinear Schr\"{o}dinger
equations attracted a great deal of attention from mathematicians in the past decades, see \cite{ca2003,ss,t} for a review.
We recall some known results about blow-up solutions for the classical nonlinear Schr\"{o}dinger equation
\begin{equation}\label{0}
iu_{t}+\Delta u=\lambda |u|^{p}u.
\end{equation}
Ginibre and Velo \cite{gv} established the local well-posedness of \eqref{0} in $H^1$(
see \cite{ca2003} for a review). Glassey \cite{gl} proved the existence of blow-up solutions
 for the negative energy and $|x|u_0\in L^2$. Ogawa and Tsutsumi \cite{ot} proved the existence
  of blow-up solutions in radial case without the restriction $|x|u_0\in L^2$. Weinstein
   \cite{we}, Zhang \cite{zj} obtained the sharp conditions of global existence for $L^2$-critical
 and $L^2$-supercritical nonlinearities. Moreover, for the $L^2$-critical nonlinearity, Weinstein
 \cite{we1} studied the structure and formation of singularity of blow-up solutions with
 critical mass by the concentration compact principle: the blow-up solution is close to
 the ground state in $H^1$ up to scaling and phase parameters, and also translation in
 the non-radial case.
%We see that the blow-up solution has the same shape as the ground state solution.
%Then, Merle and Tsutsumi [24,36] showed the $L^2$-concentration phenomenon of blow-up solutions.
% Merle [25,26]constructed the exact blow-up solutions with critical mass by the conformal invariance and compact results.
Applying the variational methods, Merle and Rapha\"{e}l \cite{m2} improved Weinstein's results and obtained the sharp decomposition of blow-up solutions with small super-critical mass. By this sharp decomposition and spectral properties, Merle and Rapha\"{e}l \cite{m1,m2,m3,m4} obtained a large body of breakthrough works, such as sharp blow-up rates, profiles, etc. Hmidi and Keraani \cite{ke} established the profile decomposition of bounded sequences in $H^1$ and gave a new and simple proof for some dynamical properties of
blow-up solutions in
$H^1$. These results have been generalized to other kinds of Schr\"{o}dinger equations, see \cite{f3,f2,fzsjmaa,gq,lz,zz,z1,z2,z3}.

In this paper, we will investigate blow-up solutions of the nonlinear Schr\"{o}dinger equation with combined power-type nonlinearities
\begin{equation}\label{e}
\left\{
\begin{array}{l}
iu_{t}+\Delta u=\lambda_1|u|^{p_1}u+\lambda_2|u|^{p_2}u, \\
u(0,x) = u_0 (x),%
\end{array}%
\right.
\end{equation}
where $u(t,x):[0,T^*)\times \mathbb{R}^N \rightarrow \mathbb{C}$ is a complex valued function and $0<T^*\leq \infty$, $0<p_2<p_1< \frac{4}{N-2}$($0<p_2<p_1\leq\infty$ if $N=1$, $0<p_2<p_1<\infty$ if $N=2$). This equation arises as  the leading-order model for propagation of intense laser beams in an isotropic bulk medium, see Section 32.1 in \cite{fibook} for a detailed explanation. Equation \eqref{e} can also be considered as a simplified model resulting from the expansion of the nonlinear Schr\"{o}dinger equation with saturated nonlinearity, which is relevant in the description of Bose superfluids at zero temperature, in the Hartree approximation, see \cite{ss}.

In \cite{tao}, Tao et al. undertook a comprehensive study for \eqref{e}.
More precisely, they addressed questions
related to local and global well-posedness, finite time blow-up, and asymptotic
behaviour.
This equation has Hamiltonian
\begin{equation}\label{h}
E(u(t)):=\frac{1}{2}\int_{\mathbb{R}^N} |\nabla u(t,x)|^2 dx+\frac{\lambda_1}{p_1+2}\int_{\mathbb{R}^N}
|u(t,x)|^{p_1+2}dx +\frac{\lambda_2}{p_2+2}\int_{\mathbb{R}^N}
|u(t,x)|^{p_2+2}dx.
\end{equation}
But there
is no scaling invariance for this equation when $p_1\neq p_2$.

About the existence of blow-up solutions, they proved the following theorem.

\textbf{Theorem A.} Let $u_0\in \Sigma:=\{u\in H^1,~xu\in L^2\}$, $\lambda_1<0$, and $\frac{4}{N}<p_1\leq \frac{4}{N-2}$ with $N\geq 3$. Let $y_0:=Im \int r\bar{u}_0\partial_ru_0dx$ denote the weighted mass current and assume $y_0>0$. Then,
blow-up occurs in each of the following three cases:

1) $\lambda_2>0$, $0<p_2<p_1$, and $E(u_0)<0$;

2) $\lambda_2<0$, $\frac{4}{N}<p_2<p_1$, and $E(u_0)<0$;

3) $\lambda_2<0$, $0<p_2\leq \frac{4}{N}$, and $E(u_0)+CM(u_0)<0$ for some suitably large constant
C.

More precisely, in any of the above cases there exists $0<T^*\leq C\frac{\|xu_0\|_{L^2}^2}{y_0}$
such that
\[
\lim_{t\rightarrow T^*}\|\nabla u(t)\|_{L^2}=\infty.
\]

As far as we know, when $\lambda_1<0$, $\lambda_2>0$, $p_1=\frac{4}{N}$, $0<p_2< \frac{4}{N}$, the existence of blow-up solutions of \eqref{e} has not been proved yet. In this paper, we
first prove the existence of blow-up solutions, and then give the sharp threshold mass
 of blow-up and global existence in this case. Moreover, we will investigate some dynamical properties of
blow-up solutions of \eqref{e} with $L^2$-critical
or $L^2$-supercritical nonlinearity, including $L^2$-concentration, $\dot{H}^{s_c}$-concentration, $L^{p_c}$-concentration, blow-up rates, and limiting profile.

To solve these problems, we mainly use the ideas from Hmidi and Keraani \cite{ke} and Guo \cite{gq}. The dynamics of blow-up solutions for the $L^2$-critical and $L^2$-supercritical nonlinear Schr\"{o}dinger equation \eqref{0} has been discussed in Hmidi and Keraani \cite{ke} and Guo \cite{gq}, respectively. In these papers, the study of dynamics of blow-up solutions relies heavily on the scale invariance of \eqref{0}. Hence, the study of dynamics of blow-up solutions for \eqref{e}, which has no the scale invariance, is of particular interest. First, we prove the existence of blow-up solutions and find the sharp threshold mass
$\|Q\|_{L^2}$ of blow-up and global existence for \eqref{e}, where $Q$ is the ground state
solution of \eqref{ell}. Then, in order to overcome the loss of scale invariance, we use the ground state solution $Q$ of \eqref{ell}
to describe the dynamical behavior of blow-up solutions to \eqref{e}. In the $L^2$-supercritical case, we choose the ground states of equations \eqref{e1} and \eqref{e2} to describe some concentration properties of blow-up solutions to \eqref{e}.

%However, compared with the nonlinear Schr\"{o}dinger equation, there are two major difficulties in the analysis of blow-up solutions of \eqref{1.1}. One is the nonlinearity containing
%the singular integral operator $E$; the other is the loss of scaling invariance to equation \eqref{1.1} with $p\neq 2$, which destroys the balance between $|u|^p u$ and $E(|u|^2)u$. Because
%of the singular integral operator $E$ in the structure of the nonlinearity, we have to use different techniques in the proof of our results. For example, we use the profile decomposition to compute the optimal constants of the corresponding
%generalized Gagliardo-Nirenberg inequalities \eqref{3.1'} and \eqref{3.11}.
%Since there
%is no scaling invariance,
%we choose the ground states of equations \eqref{e3} and \eqref{e4} to describe some
% concentration properties of the blow-up solutions to \eqref{e} respectively.

This paper is organized as follows: in Section 2, we present some preliminaries. In section 3, we first prove the existence of blow-up solutions of \eqref{e}, and then give the sharp threshold mass of blow-up and global existence. In section 4, we will consider some dynamical properties of
blow-up solutions of \eqref{e} with $p_1=\frac{4}{N}$ and $0<p_2<\frac{4}{N}$, including $L^2$-concentration, blow-up rate, and limiting profile. In section 5, we will obtain some concentration properties of
blow-up solutions of \eqref{e} with $\frac{4}{N}<p_1<\frac{4}{N-2}$ and $0<p_2<p_1$.

\textbf{Notation.}
Throughout this paper, we use the following
notation. $C> 0$ will stand for a constant that may be different from
line to line when it does not cause any confusion.
 We often use the abbreviations
$ L^{r}=L^{r}(\mathbb{R}^{N})$, $H^s=H^s(\mathbb{R}^N)$ in what
follows.
%\emph{we denote the Sobolev space $H^s(\mathbb{R}^N)$ as $H^s$ and the space $L^q(\mathbb{R}^N)$
%as $L^q$ with their norms denoted by $\|\cdot\|_{H^s}$
% and $\|\cdot\|_{L^q}$ respectively for short.}
 $\Sigma:=\{u\in H^1,~xu\in L^2\}$ denotes the energy space
equipped with the norm $\|u\|_{\Sigma}:=\|u\|_{H^1}+\|xu\|_{L^2}$.
 For $s \in \mathbb{R}$, the pseudo-differential operator $(-\Delta)^{s}$ is defined by $\widehat{(-\Delta)^{s}}f(\xi)=|\xi|^{2s}\hat{f}(\xi)$, where $\hat{}$ denotes Fourier transform. We also use
homogeneous Sobolev space $\dot{H}^s(\mathbb{R}^{N})=\{u \in
\mathcal{S}^\prime(\mathbb{R}^{N}) ; \int |\xi|^{2s}|\hat{f}(\xi)|^2d\xi<\infty \}$ with its norm defined by $\|f\|_{\dot{H}^s}=\|(-\Delta)^{s/2}f\|_{L^2}$, where $\mathcal{S}^\prime(\mathbb{R}^{N})$
denotes the space of tempered distribution on $\mathbb{R}^{N}$. In particular, we use notation: $s_c=\frac{N}{2}-\frac{2}{p_1}$ and $p_c=\frac{Np_1}{2}$. Therefore, it follows from the Sobolev embedding that $\dot{H}^{s_c}(\mathbb{R}^N)\hookrightarrow L^{p_c}(\mathbb{R}^N)$.

\section{Preliminaries}
Firstly, let us recall the local theory for the initial value problem \eqref{e} established in \cite{ca2003}.
\begin{proposition}
Let $u_0 \in H^{1}$, $0<p_1,p_2<\frac{4}{N-2}$($0<p_1,p_2<\infty$ if $N=1$, $0<p_1,p_2<\infty$ if $N=2$). Then, there exists $T = T(\|u_0
\|_{H^1})$ such that \eqref{e} admits a unique solution $u\in C([0,T],H^1)$. Let $[0,T^{\ast })$ be the maximal time interval on which the
solution $u$ is well-defined, if $T^{\ast }< \infty $, then $\|
u(t)\| _{H^{1}}\rightarrow \infty $ as $t\uparrow T^{\ast } $. Moreover, for all $0\leq t<T^*$, the solution
$ u(t)$ satisfies the following conservation of mass and energy
\[
\|u(t)\|_{L^2}=\|u_0\|_{L^2},
\]
\[
E(u(t))=E(u_0 ),
\]
 where $E(u(t))$ defined by \eqref{h}.
\end{proposition}
For more specific results concerning the Cauchy problem \eqref{e}, we refer
the reader to \cite{ca2003}. In addition, by some basic calculations, we have the following
result(see \cite{ca2003}).

\begin{lemma}
Assume that $u_0 \in \Sigma $, and the corresponding
solution $u$ of \eqref{e} exists on the interval $[0,T^*)$. Then, for all $t\in [0,T^*)$, it follows $u(t) \in \Sigma$. Moreover, let $J(t)=\int_{\mathbb{R}^N} |xu(t,x)|^2dx$, then
\begin{equation}\label{l1}
J'(t)=-4Im\int_{\mathbb{R}^N} u(t,x)x\cdot \nabla \bar{u}(t,x)dx,
\end{equation}
and
\begin{equation}\label{l2}
J''(t)=8\int_{\mathbb{R}^N} |\nabla u(t,x)|^2dx+\frac{4N\lambda_1p_1}{p_1+2}\int_{\mathbb{R}^N} |u(t,x)|^{p_1+2}dx+\frac{4N\lambda_2p_2}{p_2+2}\int_{\mathbb{R}^N} |u(t,x)|^{p_2+2}dx.
\end{equation}
\end{lemma}
Finally, we recall the following useful result of M. Weinstein \cite{we} relating the ground states of \eqref{ell} with the best constant in a Gagliardo-Nirenberg inequality.
\begin{lemma}\cite{we}
Let $p=\frac{4}{N}$ and $Q$ be the ground state solution of the following elliptic equation
\begin{equation}\label{ell}
-\Delta Q+Q=|Q|^{p+2}Q~~~in~~\mathbb{R}^N.
\end{equation}
 It follows that the optimal constant in the Gagliardo-Nirenberg inequality
\begin{equation}\label{gn}
\frac{1}{p+2}\|u\|_{L^{p+2}}^{p+2} \leq \frac{C}{2}\|u\|_{L^2}^p\|\nabla u\|_{L^2}^2,
\end{equation}
is $C=\|Q\|_{L^2}^{-p}$.
\end{lemma}
\textbf{Remark.} When $N=1$, this lemma was proved by Nagy in \cite{na}.

%Finally , we shall give the profile decomposition of bounded sequences
%in $\dot{H}^{\frac{1}{2}}\cap \dot{H}^1$ proposed by Guo \cite{gq}, which is important
%to study the variational characteristic of the ground state.
%
%\begin{proposition}
%Let $\{v_n\}_{n=1}^{\infty}$ be a bounded sequence in $\dot{H}^{s_c}\cap \dot{H}^1$. Then, there exist a subsequence of $\{v_n\}_{n=1}^{\infty}$ (still denoted by $\{v_n\}_{n=1}^{\infty}$), a family $\{x^j\}_{j=1}^{\infty}$ of sequences in $\mathbb{R}^3$ and a sequence $\{V^j\}_{j=1}^{\infty}$ in $\dot{H}^{s_c}\cap \dot{H}^1$ such that
%
%(i) for every $k\neq j$, $|x_n^k-x_n^j|\rightarrow +\infty$, as $n\rightarrow \infty$;
%
%(ii) for every $l\geq 1$ and every $x\in \mathbb{R}^3$, it follows
%
%\begin{equation}\label{3.a}
%v_n(x)=\sum_{j=1}^{l}V^j(x-x_n^j)+v_n^l(x),
%\end{equation}
%with
%\[
%\limsup_{n\rightarrow \infty}\|v_n^l\|_{L^q}\rightarrow 0 ~as ~l\rightarrow \infty
%\]
%for every $q \in (\frac{3p}{2},6)$. Moreover, we have, as $n\rightarrow \infty$,
%\begin{equation}\label{3.b}
%\|v_n\|_{\dot{H}^s}^2=\sum_{j=1}^{l}\|V^j\|_{\dot{H}^s}^2+\|v_n^l\|_{\dot{H}^s}^2+\circ(1),~for~ any~ s\in [s_c,1],
%\end{equation}
%\begin{equation}\label{3.c}
%\int_{\mathbb{R}^N} |\sum_{j=1}^l V^j(x -x_n^j)|^q dx=\sum_{j=1}^l\int_{\mathbb{R}^N} |V^j(x -x_n^j)|^q dx +\circ(1),
%\end{equation}
%\end{proposition}

\section{The sharp threshold mass of blow-up and global existence}

By the local well-posedness theory of the nonlinear Schr\"{o}dinger equation, the solution of \eqref{e} with small initial data exists globally, and for some large initial data, the solution may blow up in finite time.
Thus, whether there exists a sharp threshold of blow-up and global existence for \eqref{e} is of particular interest. On the other hand, the following problems are very important
from the view-point of physics. Under what
conditions will the condensate become unstable to collapse (blow-up)? And under what
conditions will the condensate exist for all time (global existence)? Especially the
sharp thresholds for blow-up and global existence are pursued strongly (see \cite{ca2003,f1,ss,we,zj,zz} and their references).

To solve this problem for \eqref{e}, there exists two major difficulties. One is the loss of scale invariance for \eqref{e}; the other is that the second order derivative of $J(t)=\int_{\mathbb{R}^N} |xu(t,x)|^2dx$ is the following form:
\begin{equation*}\label{l2}
J''(t)=16E(u_0)+\frac{4Np_2-16}{p_2+2}\int_{\mathbb{R}^N} |u(t,x)|^{p_2+2}dx.
\end{equation*}
Because $\int_{\mathbb{R}^N} |u(t,x)|^{p_2+2}dx$ is a positive uncertain function, which may be unbounded with respect to time $t$, it is hard to choose $E(u_0)$ to ensure the existence of blow-up solutions.

In the following theorem, by using the scaling argument and the inequality \eqref{gn}, we obtain the existence of blow-up solutions for \eqref{e} and the sharp threshold mass of blow-up and global existence for \eqref{e}.

\begin{theorem}
Let $u_0\in H^1$, $\lambda_1=-1$, $\lambda_2=1$, $p_1=\frac{4}{N}$ and $0<p_2<\frac{4}{N}$.   Assume that $Q$ is the ground state solution of \eqref{ell}. Then, we have the following sharp threshold mass of blow-up and global existence.

(i) If $\|u_0\|_{L^2}<\|Q\|_{L^2}$, then the solution of \eqref{e} exists
globally.

(ii) If the initial data $u_0=c\rho^{\frac{N}{2}} Q(\rho x)$ satisfies $|x|u_0\in L^2$, where the complex number
$c$ satisfying $|c|\geq 1$, and the real number $\rho >0$, then the
 solution $u$ of \eqref{e} with initial data $u_0$ blows up in finite time.
\end{theorem}
\textbf{Remark:} From Remark 6.6.2 in \cite{ca2003}, we infer that the critical value about the initial data for global existence of \eqref{0} and \eqref{e} is the same.
\begin{proof}
(i) We deduce from \eqref{h} and \eqref{gn} that
\begin{align*}
E(u_0)=E(u(t)) & =\frac{1}{2}\int_{\mathbb{R}^N} |\nabla u(t,x)|^2 dx-\frac{1}{p_1+2}\int_{\mathbb{R}^N}
|u(t,x)|^{p_1+2}dx +\frac{1}{p_2+2}\int_{\mathbb{R}^N}
|u(t,x)|^{p_2+2}dx\nonumber\\& \geq \left(\frac{1}{2}-\frac{\|u_0\|_{L^2}^{p_1}}{2\|Q\|_{L^2}^{p_1}}\right)\|\nabla u(t)\|_{L^2}^2.
\end{align*}
Due to $\|u_0\|_{L^2}<\|Q\|_{L^2}$, we have that $\|\nabla u(t)\|_{L^2}$ is uniformly bounded for all time $t$. Therefore,
(i) follows from the conservation of mass and Proposition 2.1.

(ii) Since $|x|u_0\in L^2$, $J(t)=\int_{\mathbb{R}^N} |xu(t,x)|^2dx$ is well-defined, and it follows from Lemma 2.2 that
\begin{equation}\label{j11}
J''(t)=16E(u_0)+\frac{4Np_2-16}{p_2+2}\int_{\mathbb{R}^N} |u(t,x)|^{p_2+2}dx.
\end{equation}
 By the
definition of initial data $u_0(x)=c\rho^{\frac{N}{2}} Q(\rho x)$ and the Poho\u{z}aev identity for equation \eqref{ell}, i.e., $\frac{1}{2}\|\nabla Q\|_{L^2}^2=\frac{1}{p_1+2}\|Q\|^{p_1+2}_{L^{p_1+2}}$, we deduce that
\begin{align}\label{h1}
E(u_0)&=\frac{|c|^2\rho^2}{2}\int_{\mathbb{R}^N} |\nabla Q(x)|^2 dx-\frac{|c|^{p_1+2}\rho^2}{p_1+2}\int_{\mathbb{R}^N}
|Q(x)|^{p_1+2}dx +\frac{|c|^{p_2+2}\rho^{\frac{N}{2}p_2}}{p_2+2}\int_{\mathbb{R}^N}
|Q(x)|^{p_2+2}dx\nonumber \\
&=-\frac{|c|^2\rho^2}{2}(|c|^{p_1}-1)\|\nabla Q\|_{L^2}^2+\frac{|c|^{p_2+2}\rho^{\frac{N}{2}p_2}}{p_2+2}\int_{\mathbb{R}^N}
|Q(x)|^{p_2+2}dx.
\end{align}
Now, taking $\rho$ such that
\[
\frac{2|c|^{p_2}\|Q\|^{p_2+2}_{L^{p_2+2}}}{(p_2+2)(|c|^{p_1}-1)\|\nabla Q\|_{L^2}^2}< \rho^{2-\frac{N}{2}p_2}.
\]
This implies $E(u_0) < 0$. It follows from \eqref{j11} that $J''(t)<16 E(u_0) < 0$.
By the standard concave argument, the solution
$u$ of \eqref{e} with the initial data $u_0$ blows up in finite time.
\end{proof}
\section{Dynamic of blow-up solutions in the case of $L^2$-critical}
In this section, we investigate some dynamical properties of blow-up solutions for \eqref{e} with $\lambda_1=-1$, $\lambda_2=1$, $p_1=\frac{4}{N}$, and $0<p_2<\frac{4}{N}$. These results are closed related to the results obtained for the classical nonlinear Schr\"{o}dinger equation \eqref{0} with $p=\frac{4}{N}$ by Hmidi and Keraani \cite{ke}. For this aim, we firstly recall the following refined compactness result which can be proved by using the profile
decomposition of bounded sequences in $H^1$ and the inequality \eqref{gn}, see \cite{ke}.
\begin{lemma}
Let $\{u_n\}_{n=1}^{\infty}$ be a bounded sequence in $H^1$, such that
\begin{equation*}
\limsup_{n\rightarrow \infty}\|\nabla u_n\|_{L^2}\leq M,~~~\limsup_{n\rightarrow \infty}\|u_n\|_{L^{4/N+2}}\geq m>0.
\end{equation*}
Then, there exist $V\in H^1$ and $\{x_n\}_{n=1}^{\infty}\subset \mathbb{R}^N$ such that, up to a subsequence,
\begin{equation*}
u_n(\cdot+x_n)\rightharpoonup V~~weakly~in~H^1
\end{equation*}
with
\begin{equation*}
\|V\|_{L^2}\geq \left(\frac{N}{N+2}\right)^{\frac{N}{4}}\frac{m^{N/2+1}}{M^{N/2}}\|Q\|_{L^2}.
\end{equation*}
where $Q$ is the ground state solution of \eqref{ell}.
\end{lemma}
\begin{theorem}($L^2$-concentration)
Let $u_0\in H^1$, $\lambda_1=-1$, $\lambda_2=1$, $p_1=\frac{4}{N}$, and $0<p_2<\frac{4}{N}$. If the solution $u$ of \eqref{e} blows up in finite time $T^*>0$.
Let $a(t)$ be a real-valued nonnegative function defined on $[0,T^*)$ satisfying $a(t)\|\nabla u(t)\|_{L^2}\rightarrow \infty $ as $t\rightarrow T^*$. Then there exists $x(t)\in \mathbb{R}^{N}$ such that
\begin{equation}\label{41}
\liminf_{t\rightarrow T^*}\int_{|x-x(t)|\leq a(t)}|u(t,x)|^2dx\geq \int_{\mathbb{R}^N} |Q(x)|^2dx.
\end{equation}
where $Q$ is the ground state solution of \eqref{ell}.

%(2) If $\mathcal{A}$ is the set of weak $L^2$-limit points of $u(t)$ as $t\rightarrow T^*$, then
%\begin{equation}\label{42}
%\|V\|_{L^2}^2\leq \|u_0\|_{L^2}^2-\|R\|_{L^2}^2~~for~all~V\in \mathcal{A}.
%\end{equation}

%(4) Assume that the initial mass is small super-critical: there exists $\alpha_1>0$ such that for all $0<\alpha'\leq \alpha_1$, there exists $\delta(\alpha')>0$ with $\delta(\alpha')\rightarrow 0$ as $\alpha'\rightarrow 0$, and for any $u_0\in H^1$ and $0 <\alpha (u_0):=\|u_0\|_{L^2}^2-\|R\|_{L^2}^2<\alpha'$. Then, there exist functions $y(t)\in \mathbb{R}^N$ and $\theta(t)\in \mathbb{R}$ such that when $t\rightarrow T^*$,
%\begin{equation}\label{43}
%\|\rho^{N/2}(t)u(t,\rho(t)(\cdot+y(t)))e^{i\theta(t)}-R\|_{H^1}\leq \delta(\alpha').
%\end{equation}

\end{theorem}
\textbf{Remark.} Theorem 4.2 gives the $L^2$-concentration and rate of $L^2$-concentration of blow-up solutions of \eqref{e}. Indeed, we can choose $a(t)=\frac{1}{{\|\nabla u(t)\|}^{1-\delta}_{L^2}}$ with $0<\delta<1$. It is obvious that $\lim_{t\rightarrow T^*} a(t)=0$ and $a(t)$ satisfies the assumption in Theorem 4.2. Applying Theorem 4.2, if $u$ is a blow-up solution of \eqref{e} and $T^*$ its blow-up time, then for every $r>0$, there exists a function $x(t)\in\mathbb{R}^N$ such that
\[
\liminf_{t\rightarrow T^*}\int_{|x-x(t)|\leq r}|u(t,x)|^2dx\geq\int_{\mathbb{R}^N}|Q(x)|^2dx.
\]
Meanwhile, it follows from the choice of $a(t)$ that for any function $a(t)\leq \frac{1}{{\|\nabla u(t)\|}^{1-\delta}_{L^2}}$, \eqref{41} holds, which implies that the rate of $L^2$-concentration of blow-up solutions of \eqref{e} is $\frac{1}{{\|\nabla u(t)\|}^{1-\delta}_{L^2}}$ with $0<\delta<1$.

\begin{proof}
 Set
\[
\rho(t)=\|\nabla Q\|_{L^2}/\|\nabla u(t)\|_{L^2}~~and~~v(t,x)=\rho^{\frac{N}{2}}(t)u(t,\rho(t) x).
\]
Let $\{t_n\}_{n=1}^\infty$ be an any time sequence such that $t_n\rightarrow T^*$, $\rho_n:=\rho(t_n)$ and $v_n(x):=v(t_n,x)$.
Then, the sequence $\{v_n\}$
satisfies
\begin{equation}\label{44}
\|v_n\|_{L^2}=\|u(t_n)\|_{L^2}=\|u_0\|_{L^2},~~\|\nabla v_n\|_{L^2}=\rho_n\|\nabla u(t_n)\|_{L^2}=\|\nabla Q\|_{L^2}.
\end{equation}
Observe that
\begin{align}\label{45}
 H(v_n):=&\frac{1}{2}\int_{\mathbb{R}^N} |\nabla v_n(x)|^2dx-\frac{1}{p_1+2}\int_{\mathbb{R}^N}  |v_n(x)|^{p_1+2}dx
\nonumber\\
  =&\rho_n^2\left(\frac{1}{2}\int_{\mathbb{R}^N} |\nabla u(t_n,x)|^2dx-\frac{1}{p_1+2}\int_{\mathbb{R}^N}  |u(t_n,x)|^{p_1+2}dx\right)\nonumber\\
  =&\rho_n^2\left(E(u_0)-\frac{1}{p_2+2}\int_{\mathbb{R}^N} |u(t_n,x)|^{p_2+2}dx\right).
\end{align}
Thus, applying the following Gagliardo-Nirenberg inequality
\begin{equation*}
\int_{\mathbb{R}^N} |u(x)|^{p_2+2}dx\leq C\|u\|_{L^2}^{p_2+2-\frac{Np_2}{2}}\|\nabla u\|_{L^2}^{\frac{Np_2}{2}},~~for~0<p_2<\frac{4}{N},
\end{equation*}
we deduce that
\begin{align}\label{xx45}
 |H(v_n)|&\leq\rho_n^2\left(|E(u_0)|+\frac{1}{p_2+2}\int_{\mathbb{R}^N} |u(t_n,x)|^{p_2+2}dx\right)\nonumber\\&\leq \frac{|E(u_0)|\|\nabla Q\|_{L^2}^2}{\|\nabla u(t_n)\|_{L^2}^2}+C\frac{\|\nabla Q\|_{L^2}^2\|\nabla u(t_n)\|_{L^2}^{\frac{Np_2}{2}}}{\|\nabla u(t_n)\|_{L^2}^2}\nonumber\\&
 \rightarrow 0~~as~n\rightarrow\infty,
\end{align}
 which implies
$\int_{\mathbb{R}^N} |v_n(x)|^{p_1+2}dx\rightarrow (2/N+1)\|\nabla Q\|_{L^2}^2$.

Set $m^{p_1+2}=(2/N+1)\|\nabla Q\|_{L^2}^2$ and $M=\|\nabla Q\|_{L^2}$. Then it follows from Lemma 4.1 that there exist $V\in H^1$ and $\{x_n\}_{n=1}^\infty \subset \mathbb{R}^N$ such that, up to a subsequence,
\begin{equation}\label{46'}
v_n(\cdot +x_n)=\rho_n^{N/2}u(t_n,\rho_n(\cdot + x_n))\rightharpoonup V~~weakly~in~H^1
\end{equation}
with
\begin{equation}\label{46}
\|V\|_{L^2}\geq \|Q\|_{L^2}.
\end{equation}
Note that
\[
\frac{a(t_n)}{\rho_n}=\frac{a(t_n)\|\nabla u(t_n)\|_{L^2}}{\|\nabla Q\|_{L^2}}\rightarrow \infty,~~as~n\rightarrow\infty.
\]
Then for every $r>0$, there exists $n_0>0$ such that for every $n>n_0$, $r\rho_n<a(t_n)$. Therefore, using
\eqref{46'}, we obtain
\begin{align*}\label{45}
\liminf_{n\rightarrow \infty}\sup_{y\in \mathbb{R}^N}\int_{|x-y|\leq a(t_n)}|u(t_n,x)|^2dx&\geq \liminf_{n\rightarrow \infty}\sup_{y\in \mathbb{R}^N}\int_{|x-y|\leq r\rho_n}|u(t_n,x)|^2dx
\nonumber\\
  &\geq \liminf_{n\rightarrow \infty}\int_{|x-x_n|\leq r\rho_n}|u(t_n,x)|^2dx\nonumber\\
  &=\liminf_{n\rightarrow \infty}\int_{|x|\leq r}\rho_n^{N}|u(t_n,\rho_n(x+ x_n))|^2dx\nonumber\\
  &=\liminf_{n\rightarrow \infty}\int_{|x|\leq r}|v(t_n,x+ x_n)|^2dx\nonumber\\
  &\geq\liminf_{n\rightarrow \infty}\int_{|x|\leq r}|V(x)|^2dx,~~for~every~r>0,
\end{align*}
which means that
\[
\liminf_{n\rightarrow \infty}\sup_{y\in \mathbb{R}^N}\int_{|x-y|\leq a(t_n)}|u(t_n,x)|^2dx\geq\int_{ \mathbb{R}^N}|V(x)|^2dx.
\]
%
%Thus, for any $M>0$,
%\[
%\liminf_{n\rightarrow \infty}\int_{|x|\leq M}\rho_n^{N}|u(t_n,\rho_n(x+ x_n))|^2dx\geq \int_{|x|\leq M}|V(x)|^2dx.
%\]
%In view of the assumption  $a(t)\|\nabla u(t)\|_{L^2}\rightarrow \infty $, this implies immediately
%\begin{equation*}
%\liminf_{n\rightarrow \infty}\int_{|z-\rho_nx_n|\leq \rho_nM}|u(t_n,z)|^2dz\geq \int_{|x|\leq M}|V(x)|^2dx.
%\end{equation*}
Since the sequence $\{t_n\}_{n=1}^\infty$ is arbitrary, we obtain
\begin{equation}\label{45x}
\liminf_{t\rightarrow T^*}\sup_{y\in \mathbb{R}^N}\int_{|x-y|\leq a(t)}|u(t,x)|^2dx\geq\int_{ \mathbb{R}^N}|Q(x)|^2dx.
\end{equation}
Observe that for every $t\in [0,T^*)$, the function $g(y):= \int_{|x-y|\leq a (t)}|u(t,x)|^2dx$ is continuous on $y\in \mathbb{R}^N$ and $g(y)\rightarrow 0$ as $|y|\rightarrow \infty$. So there exists a function $x(t)\in \mathbb{R}^N$ such that for every $t\in [0,T^*)$
\begin{equation*}
 \sup_{y\in \mathbb{R}^N}\int_{|x-y|\leq a (t)}|u(t,x)|^2dx=\int_{|x-x(t)|\leq a(t)}|u(t,x)|^2dx.
\end{equation*}
This and \eqref{45x} yield \eqref{41}.
\end{proof}
In the following theorem, we study the limiting profile of blow-up solutions of \eqref{e}.
\begin{theorem}
Let $u_0\in H^1$, $\lambda_1=-1$, $\lambda_2=1$, $p_1=\frac{4}{N}$, and $0<p_2<\frac{4}{N}$.
Assume $\|u_0\|_{L^2}=\|Q\|_{L^2}$, and the corresponding solution $u$ of \eqref{e} blows up in finite time $T^*>0$, then there exist $x(t)\in \mathbb{R}^N$ and $\theta(t)\in [0,2\pi)$ such that
\begin{equation}\label{43'}
\rho^{N/2}(t)u(t,\rho(t)(\cdot+x(t)))e^{i\theta(t)}\rightarrow Q~strongly~in~H^1,~as~t\rightarrow T^*,
\end{equation}
where $\rho(t)=\frac{\|\nabla Q\|_{L^2}}{\|\nabla u(t)\|_{L^2}}$.

%(4) Assume that the initial mass is small super-critical: there exists $\alpha_1>0$ such that for all $0<\alpha'\leq \alpha_1$, there exists $\delta(\alpha')>0$ with $\delta(\alpha')\rightarrow 0$ as $\alpha'\rightarrow 0$, and for any $u_0\in H^1$ and $0 <\alpha (u_0):=\|u_0\|_{L^2}^2-\|R\|_{L^2}^2<\alpha'$. Then, there exist functions $y(t)\in \mathbb{R}^N$ and $\theta(t)\in \mathbb{R}$ such that when $t\rightarrow T^*$,
%\begin{equation}\label{43}
%\|\rho^{N/2}(t)u(t,\rho(t)(\cdot+y(t)))e^{i\theta(t)}-R\|_{H^1}\leq \delta(\alpha').
%\end{equation}

\end{theorem}

\begin{proof}
%We prove \eqref{43'} by contradiction. More precisely, our aim is to prove that for arbitrary $\varepsilon>0$ and any sequence $\{t_n\}^\infty_{n=1}$ such that $t_n\rightarrow T^*$ as $n\rightarrow \infty$, we have
%\begin{equation}\label{49}
%\|\rho^{N/2}_nu(t_n,\rho_n(\cdot+y_{n}))e^{i\theta_n}-Q(\cdot)\|_{H^1}<\varepsilon,
%\end{equation}
%where parameters: $\{\rho_n\}^\infty_{n=1}\subset \mathbb{R}^+,\{y_n\}^{\infty}_{n=1}\subset \mathbb{R}^N$, and $\{\theta_n\}^{\infty}_{n=1}\subset \mathbb{R}$. If not, then \eqref{49} is not true for some sequences $\{t_n\}^{\infty}_{n=1}$.

We use the notations in the proof of Theorem 4.2. Assume that
$\|u_0\|_{L^2}=\|Q\|_{L^2}$. Recall that we have verified that
$\|V\|_{L^2}\geq \|Q\|_{L^2}$ in the proof of Theorem 4.2. Whence
%Next, we will find a subsequence of $\{t_n\}^{\infty}_{n=1}$ such that \eqref{49} holds. Indeed, by a similar argument as \eqref{46'}, we can obtain
%\begin{equation}\label{49''}
%v_n(\cdot+y_n)e^{i\theta_n}\rightharpoonup V_1~ weakly ~in ~L^2~ with ~\|V_1\|_{L^2}\geq\|Q\|_{L^2}.
%\end{equation}
%By lower semi-continuity of the $L^2$ norm that
\[
\|Q\|_{L^2}\leq\|V\|_{L^2}\leq  \liminf_{n\rightarrow \infty }\|v_n\|_{L^2}=\liminf_{n\rightarrow \infty }\|u(t_n)\|_{L^2}=\|u_0\|_{L^2}=\|Q\|_{L^2},
\]
and then,
\begin{equation}\label{49'}
\lim_{n\rightarrow \infty }\|v_n\|_{L^2}=\|V\|_{L^2}=\|Q\|_{L^2},
\end{equation}
which implies
\[
v_n(\cdot+x_n)\rightarrow V~ strongly ~in ~L^2~as~n\rightarrow \infty.
\]
We infer from the inequality \eqref{gn} that
\[
\|v_n(\cdot+x_{n})-V\|^{p_1+2}_{L^{p_1+2}}\leq C\|v_n(\cdot+x_{n})-V\|^{p_1}_{L^2}\|\nabla (v_n(\cdot+x_{n})
-V)\|^2_{L^2}.
\]
From $\|\nabla v_n(\cdot+x_{n})\|_{L^2}\leq C$, we get
\[
v_n(\cdot+x_{n})\rightarrow V~~in ~L^{p_1+2}~as~n\rightarrow \infty.
\]

Next, we will prove that $v_n(\cdot+x_{n})$ converges to $V$ strongly in $H^1$. For this aim, we estimate as follows:
\begin{align}\label{410}
0 =&\lim_{n\rightarrow \infty }H(v_n)
\nonumber\\
  =&\frac{1}{2}\int_{\mathbb{R}^N}|\nabla Q(x)|^2dx-\frac{1}{p_1+2}\lim_{n\rightarrow \infty}\int_{\mathbb{R}^N} |v_n(x)|^{p_1+2}dx\nonumber\\
  =&\frac{1}{2}\int_{\mathbb{R}^N}|\nabla Q(x)|^2dx-\frac{1}{p_1+2}\int_{\mathbb{R}^N}|V(x)|^{p_1+2}dx.
\end{align}
Thus, we infer from the inequality \eqref{gn} that
\begin{equation}\label{411h}
\frac{1}{2}\int_{\mathbb{R}^N}|\nabla Q(x)|^2dx=\frac{1}{p_1+2}\int_{\mathbb{R}^N}|V(x)|^{p_1+2}dx\leq\frac{1}{2}\frac{\| V\|_{L^2}^{p_1}}{\|Q\|_{L^2}^{p_1}}\|\nabla V\|_{L^2}^2=\frac{1}{2}\|\nabla V\|_{L^2}^2.
\end{equation}
On the other hand, we deduce from \eqref{44} that $ \|\nabla V\|_{L^2}\leq \liminf_{n\rightarrow \infty}\|\nabla v_n(\cdot+x_{n})\|_{L^2}=\|\nabla Q\|_{L^2}$. Hence, we have $\| Q\|_{H^1}=\|V\|_{H^1}$ and
\begin{equation}\label{411}
v_n(\cdot+x_{n})\rightarrow V~ strongly ~in ~H^1~as~n\rightarrow \infty.
\end{equation}
%Now, collect the properties of $V_1(x)$:
%\[
%H(V_1)=0,~~\|V_1\|_{L^2}=\|R\|_{L^2}~and~\|\nabla R\|_{L^2}=\|\nabla V_1\|_{L^2}.
%\]
This and \eqref{411h} imply that
\[
H(V)=\frac{1}{2}\int_{\mathbb{R}^N}|\nabla V(x)|^2dx-\frac{1}{p_1+2}\int_{\mathbb{R}^N}|V(x)|^{p_1+2}dx=0.
\]
Up to now, we have verified that
\[
\|V\|_{L^2}=\|Q\|_{L^2},~\|\nabla V\|_{L^2}=\|\nabla Q\|_{L^2}~ and ~H(V)=0.
\]
The variational characterization of the ground state implies that
\[
V(x)=e^{i\theta}Q(x+x_0)~for~some~\theta\in [0,2\pi),~x_0\in \mathbb{R}^N
\]
and
\[
\rho^{N/2}_nu(t_n,\rho_n(\cdot+x_{0}))\rightarrow e^{i\theta}Q(\cdot+x_0)~strongly~in~H^1~as~n\rightarrow \infty.
\]
Since the sequence $\{t_n\}_{n=1}^\infty$ is arbitrary, we infer that there are two functions $x(t)\in \mathbb{R}^N$ and $\theta(t)\in [0,2\pi)$ such that
\[
\rho^{N/2}(t)e^{i\theta(t)}u(t,\rho(t)(x+x(t)))\rightarrow Q~strongly~in~H^1~as~t\rightarrow T^*.
\]
\end{proof}

In the following theorem, we will prove that the blow-up solution $|u(t,x)|^2$ like a $\delta$-function as $t\rightarrow T^*$ at the
point $x = x_0$, which implies that the point $x_0$ concentrates all mass of blow-up
solution of \eqref{e}.

\begin{theorem}\label{THlimit}
Let $u_0 \in \Sigma$, $\lambda_1=-1$, $\lambda_2=1$, $p_1=\frac{4}{N}$, and $0<p_2<\frac{4}{N}$. If the solution $u$ of \eqref{e} blows up in finite time $T^*>0$ and $\|u_0\|_{L^2}=\|Q\|_{L^2}$, then there exists $x_0\in \mathbb{R}^N$ such that
 \begin{equation}\label{71}
 |u(t,x)|^2\rightarrow \|Q\|_{L^2}^2\delta_{x_0}
 \end{equation}
 in the sense of distribution as $t\rightarrow T^*$.
\end{theorem}
\begin{proof}
According to Theorem 4.2, it follows that for all $r>0$
\begin{equation}
\liminf_{t\rightarrow T^*} \int_{|x-x(t)|<r} |u(t, x)|^2dx\geq \|Q\|_{L^2}^2.
\end{equation}
This, together with the conservation of mass $\|u(t)\|_{L^2}^2=\|u_0\|_{L^2}^2=\|Q\|_{L^2}^2$, implies that
 for all $r>0$
\begin{equation*}
\liminf_{t\rightarrow T^*} \int_{|x-x(t)|<r} |u(t, x)|^2dx= \|Q\|_{L^2}^2.
\end{equation*}
This yields
 \begin{equation}\label{73}
 |u(t,x+x(t))|^2\rightarrow \|Q\|_{L^2}^2\delta_{x=0}.
 \end{equation}

By using the inequality \eqref{gn}, for any $\varepsilon >0$ and any real-valued function $\theta$ defined on $\mathbb{R}^N$, we have
\begin{align*}
H(e^{\pm i\epsilon \theta}u) &\geq \frac{1}{2}\int_{\mathbb{R}^N} | \nabla (e^{\pm i\epsilon \theta}u)|^2dx\left( 1-\frac{\|u\|_{L^2}^{p_1}}{\|Q\|_{L^2}^{p_1}} \right)
= 0.
\end{align*}
Therefore,
\begin{equation*}
0\leq H(e^{\pm i\epsilon \theta}u)=\frac{\epsilon^2}{2}\int_{\mathbb{R}^N} |u|^2|\nabla \theta|^2dx\mp \epsilon Im \int_{\mathbb{R}^N} \bar{u}\nabla u\cdot \nabla \theta dx +H(u),
 \end{equation*}
which implies that
\begin{equation}\label{72}
\left|\mp Im \int_{\mathbb{R}^N} \bar{u}\nabla u\cdot \nabla \theta dx\right|\leq \left(2H(u)\int_{\mathbb{R}^N} |u|^2|\nabla \theta|^2dx\right)^{1/2}.
 \end{equation}
For any $j=1,2,\ldots,N$, it follows from \eqref{72} and $H(u(t))\leq E(u(t))=E(u_0)$ that
\begin{align*}
\left| \frac{d}{dt}\int_{\mathbb{R}^N} |u(t,x)|^2x_jdx\right| &= 2\left| \int_{\mathbb{R}^N} \bar{u}\partial_ju dx\right|
\nonumber
\\
& = 2\left| \int_{\mathbb{R}^N} \bar{u}\nabla u \nabla x_j dx\right|\nonumber
\\
& \leq 2\left(2H(u)\int_{\mathbb{R}^N} |u|^2|\nabla x_j|^2dx\right)^{1/2}\nonumber
\\
& \leq C.
\end{align*}
Let $t_m,t_k\in (0,T^*)$ be any two sequences satisfying $\lim_{m\rightarrow\infty}t_m=\lim_{k\rightarrow\infty}t_k=T^*$. Then for any $j=1,2,\ldots,N$, we have
\begin{equation*}
\left| \int_{\mathbb{R}^N} |u(t_m,x)|^2x_jdx- \int_{\mathbb{R}^N} |u(t_k,x)|^2x_jdx\right|\leq C|t_m-t_k|\rightarrow0~~as~~m,k\rightarrow \infty,
 \end{equation*}
which implies that
\begin{equation*}
\lim _{t\rightarrow T^*}\int_{\mathbb{R}^N} |u(t,x)|^2x_jdx~~exists~~for~any~j=1,2,\ldots,N.
 \end{equation*}
Set
\begin{equation*}
x_0=\|Q\|_{L^2}^{-2}\lim _{t\rightarrow T^*}\int_{\mathbb{R}^N} |u(t,x)|^2xdx.
\end{equation*}
Then
\begin{equation}\label{78}
\lim _{t\rightarrow T^*}\int_{\mathbb{R}^N} |u(t,x)|^2xdx=x_0\|Q\|_{L^2}^2.
 \end{equation}
On the other hand, we infer from Lemma 2.2 that there is a constant $c_0$ such that
\begin{equation*}
\int_{\mathbb{R}^N} |x|^2 |u(t,x)|^2dx\leq c_0.
\end{equation*}
This yields
\begin{align}\label{76}
\int_{\mathbb{R}^N} |x|^2|u(t,x+x(t))|^2dx &\leq 2\int_{\mathbb{R}^N} |x+x(t)|^2|u(t,x+x(t))|^2dx+2|x(t)|^2\int_{\mathbb{R}^N} |u(t,x+x(t))|^2dx
\nonumber
\\
& \leq 2c_0+2|x(t)|^2\|u_0\|_{L^2}^2.
\end{align}
We infer from \eqref{73}
that
\begin{equation*}
\limsup_{t\rightarrow T^*}|x(t)|^2\|Q\|_{L^2}^2=\limsup_{t\rightarrow T^*}\int_{|x|<1}|x+x(t)|^2|u(t,x+x(t))|^2dx\leq \int_{\mathbb{R}^N} |x|^2|u(t,x)|^2dx\leq c_0.
\end{equation*}
Thus,
\begin{equation}\label{75}
\limsup_{t\rightarrow T^*}|x(t)|\leq \frac{\sqrt{c_0}}{\|Q\|_{L^2}}.
\end{equation}
Combining \eqref{76} and \eqref{75}, we obtain
\begin{equation*}
\limsup_{t\rightarrow T^*}\int_{\mathbb{R}^N} |x|^2|u(t,x+x(t))|^2dx\leq C.
\end{equation*}
Hence, for any $\varepsilon >0$, there exists $R_0=R_0(\varepsilon)$ such that
\begin{equation*}
\limsup_{t\rightarrow T^*}\left|\int_{|x|\geq R_0} x|u(t,x+x(t))|^2dx\right|\leq \frac{C}{R_0}<\varepsilon.
\end{equation*}
It follows from \eqref{73} that
\begin{align}\label{77}
\limsup_{t\rightarrow T^*}\left |\int_{\mathbb{R}^N} |u(t,x)|^2xdx-x(t)\|Q\|_{L^2}^2\right| &= \limsup_{t\rightarrow T^*}\left |\int_{\mathbb{R}^N} |u(t,x)|^2(x-x(t))dx\right|
\nonumber
\\
& \leq \limsup_{t\rightarrow T^*}\left |\int_{|x|\leq R_0} |u(t,x+x(t))|^2xdx\right|+\varepsilon\nonumber
\\
& \leq \varepsilon,
\end{align}
which, together with \eqref{78} implies that $\lim_{t\rightarrow T^*}x(t)=x_0$. Therefore,
\begin{equation*}\label{79}
\limsup_{t\rightarrow T^*}\int_{\mathbb{R}^N} |u(t,x)|^2xdx=\|Q\|_{L^2}^2x_0,
\end{equation*}
and
\begin{equation*}\label{710}
 |u(t,x)|^2\rightarrow \|Q\|_{L^2}^2\delta_{x=x_0} ~in~the~sence~of~distribution~as~t\rightarrow T^*.
 \end{equation*}
\end{proof}
The following theorem gives the lower bound for the blow-up rate of blow-up solutions with critical mass $\|u_0\|_{L^2}=\|Q\|_{L^2}$.
\begin{theorem}\label{THrate}
Let $u_0 \in \Sigma$, $\lambda_1=-1$, $\lambda_2=1$, $p_1=\frac{4}{N}$, and $0<p_2<\frac{4}{N}$. If the solution $u$ of \eqref{e} blows up in finite time $T^*>0$ and $\|u_0\|_{L^2}=\|Q\|_{L^2}$, then there exists a constant $C>0$ such that
 \begin{equation}\label{81}
\|\nabla u(t)\|_{L^2}\geq \frac{C}{T^*-t},~~\forall t\in [0,T^*).
 \end{equation}
\end{theorem}
\begin{proof}
Let $h \in C_0^\infty (\mathbb{R}^N)$ be a nonnegative radial function such that
\[
h(x)=h(|x|)=|x|^2,~~if~|x|<1~~and~~|\nabla h(x)|^2\leq Ch(x).
\]
For $A>0$, we define $h_A(x)=A^2h(\frac{x}{A})$
and $g_A(t)=\int h_A(x-x_0)|u(t,x)|^2dx$ with $x_0$
defined by \eqref{79}.

From \eqref{72}, for every $t\in  [0,T^*)$, we have
\begin{align}\label{82}
\left |\frac{d}{dt}g_A(t)\right| &= 2\left |Im \sum_{j=1}^{N}\int_{\mathbb{R}^N} \bar{u}(t,x)\nabla u(t,x)\nabla h_A(x-x_0)dx\right|
\nonumber
\\
& \leq 2\sqrt{E(u_0)}\left (\int_{\mathbb{R}^N} |u(t,x)|^2|\nabla h_A(x-x_0)|^2dx\right)^{1/2}\nonumber
\\
& \leq C\sqrt{g_A(t)},
\end{align}
which implies
\begin{equation*}
\left |\frac{d}{dt}\sqrt{g_A(t)}\right|\leq C.
\end{equation*}
Integrating on both sides, we obtain
\begin{equation}\label{829}
\left |\sqrt{g_A(t)}-\sqrt{g_A(t_1)}\right|\leq C|t-t_1|.
\end{equation}
It follows from \eqref{71} that
\[
g_A(t_1)\rightarrow \|Q\|_{L^2}h_A(0)=0~~as~~t_1\rightarrow T^*.
\]
Therefore, letting $t_1\rightarrow T^*$ in \eqref{829}, it follows that
\[
g_A(t)\leq C(T^*-t)^2.
\]
Now fix $t \in [0,T^*)$ and let $A$ go to infinity, we have

\begin{equation*}
\int_{\mathbb{R}^N} |x-x_0|^2|u(t,x)|^2dx\leq C(T^*-t)^2.
\end{equation*}
Then the uncertainty principle
\begin{equation*}
\left(\int_{\mathbb{R}^N} |u(t,x)|^2dx\right)^2\leq \left(\int_{\mathbb{R}^N} |x-x_0|^2|u(t,x)|^2dx\right)\left(\int_{\mathbb{R}^N} |\nabla u(t,x)|^2dx\right),
\end{equation*}
implies a lower bound of the blow-up rate
\begin{equation*}
\|\nabla u(t)\|_{L^2}\geq \frac{C}{T^*-t},~~\forall t\in [0,T^*).
 \end{equation*}
\end{proof}

\section{The $L^2$-supercritical case}
When $\lambda_1<0$, $\lambda_2\in \mathbb{R}$, $\frac{4}{N}<p_1<\frac{4}{N-2}$ and $0<p_2<p_1$, for some large initial data, the solution may blow up in finite time.
In order to investigate some concentration properties of the blow-up solutions to \eqref{e} in this case, we need the following version of compactness lemma which comes from \cite{gq}.
\begin{lemma}
Let $\{u_n\}_{n=1}^{\infty}$ be a bounded sequence in $\dot{H}^{s_c}\cap \dot{H}^1$, such that
\begin{equation*}
\limsup_{n\rightarrow \infty}\|\nabla u_n\|_{L^2}\leq M,~~~\limsup_{n\rightarrow \infty}\|u_n\|_{L^{p_1+2}}\geq m>0.
\end{equation*}
Then, there exist $\{x^1_n\}_{n=1}^{\infty},\{x^2_n\}_{n=1}^{\infty}\subset \mathbb{R}^N$, $V_1\in\dot{H}^{s_c}\cap \dot{H}^1$, and $V_2\in L^{p_c} \cap \dot{H}^1$ such that, up to a subsequence,
\begin{equation*}
u_n(\cdot+x^1_n)\rightharpoonup V_1~~weakly~in~~\dot{H}^{s_c}\cap \dot{H}^1,
\end{equation*}
with
\begin{equation*}
\|V_1\|_{\dot{H}^{s_c}}^{p_1}\geq \frac{2m^{p_1+2}}{(p_1+2)M^2}\|Q\|_{\dot{H}^{s_c}}^{p_1},
\end{equation*}
and
\begin{equation*}
u_n(\cdot+x^2_n)\rightharpoonup V_2~~weakly~in~~L^{p_c} \cap \dot{H}^1,
\end{equation*}
with
\begin{equation*}
\|V_2\|_{L^{p_c}}^2\geq \frac{2m^{p_1+2}}{(p_1+2)M^2}\|R\|_{L^{p_c}}^2,
\end{equation*}
where $Q$ and $R$ are the solutions of the following elliptic equations
\begin{equation}\label{e1}
-\Delta Q +\frac{p_1}{2}(-\triangle)^{s_c}Q=|Q|^{p_1}Q,
\end{equation}
and
\begin{equation}\label{e2}
-\Delta R +|R|^{p_c-2}R =|R |^{p_1}R,
\end{equation}
respectively.
\end{lemma}
In the following theorem, we will use the ground states of equations \eqref{e1} and \eqref{e2} to describe some concentration properties of blow-up solutions to \eqref{e}.
This result is closed related to the result obtained for the classical nonlinear Schr\"{o}dinger equation \eqref{0} with $\frac{4}{N}<p<\frac{4}{N-2}$ by Guo \cite{gq}.
\begin{theorem}
Let $\lambda_1<0$, $\lambda_2\in \mathbb{R}$, $\frac{4}{N}<p_1<\frac{4}{N-2}$, $0<p_2< p_1$ and $u_0 \in \dot{H}^{s_c}\cap \dot{H}^1$. If the solution $u$ of \eqref{e} blows up in finite time $T^*>0$ and satisfies
\begin{equation}\label{a}
\sup_{t\in [0,T^*)}\|u(t)\|_{\dot{H}^{s_c}}< \infty.
\end{equation}
Assume that $\lambda(t)>0$ such that
\begin{equation}\label{b}
\lambda(t)\|\nabla u(t)\|_{L^2}^{\frac{1}{s_c}}\rightarrow \infty,
\end{equation}
as $t\rightarrow T^*$. Then, there exist $x_1(t),x_2(t)\in \mathbb{R}^N$ such that
\begin{equation}\label{c}
\liminf_{t\rightarrow T^*}\int_{|x-x_1(t)|\leq \lambda(t)}|(-\triangle)^{\frac{s_c}{2}}u(t,x)|^2dx\geq \|Q\|_{\dot{H}^{s_c}}^2,
\end{equation}
and
\begin{equation}\label{d}
\liminf_{t\rightarrow T^*}\int_{|x-x_2(t)|\leq \lambda(t)}|u(t,x)|^{p_c}dx\geq \|R\|_{L^{p_c}}^{p_c},
\end{equation}
where $Q$ and $R$ solve the elliptic equations (\ref{e1}) and \eqref{e2}, respectively.
\end{theorem}
\begin{proof}
 Set
\[
\rho(t)=\|\nabla Q\|_{L^2}^{\frac{1}{1-s_c}}/\|\nabla u(t)\|_{L^2}^{\frac{1}{1-s_c}}~~and~~v(t,x)=\rho^{\frac{2}{p_1}}(t)u(t,\rho(t) x).
\]
Let $\{t_n\}_{n=1}^\infty$ be an any time sequence such that $t_n\rightarrow T^*$, $\rho_n=\rho(t_n)$ and $v_n(x)=v(t_n,x)$.
Then, it follows from assumption \eqref{a} that $v_n$ satisfies $\|v_n\|_{\dot{H}^{s_c}}=\|u(t_n)\|_{\dot{H}^{s_c}}< \infty$ uniformly in $n$.
Moreover, by some direct computations, we obtain
\begin{equation*}
\|\nabla v_n\|_{L^2}=\rho_n^{1-s_c}\|\nabla u(t_n)\|_{L^2}=\|\nabla Q\|_{L^2},
\end{equation*}
and
\begin{align}\label{51}
H(v_n):= &\frac{1}{2}\int_{\mathbb{R}^N} |\nabla v_n(x)|^2dx+\frac{\lambda_1}{p_1+2}\int_{\mathbb{R}^N} |v_n(x)|^{p_1+2}dx
\nonumber\\
  =&\rho_n^{2(1-s_c)}\left(\frac{1}{2}\int_{\mathbb{R}^N} |\nabla u(t_n,x)|^2dx+\frac{\lambda_1}{p_1+2}\int_{\mathbb{R}^N} |u(t_n,x)|^{p_1+2}dx\right)\nonumber\\
  =&\rho_n^{2(1-s_c)}\left(E(u(t_n))-\frac{\lambda_2}{p_2+2}\int_{\mathbb{R}^N} |u(t_n,x)|^{p_2+2}dx\right)\nonumber\\
 = & \frac{\|\nabla Q\|_{L^2}^2}{\|\nabla u(t_n)\|_{L^2}^2}\left(E(u_0)-\frac{\lambda_2}{p_2+2}\int_{\mathbb{R}^N} |u(t_n,x)|^{p_2+2}dx\right).
\end{align}
In the following, we will prove $H(v_n)\rightarrow 0$ as $n\rightarrow\infty$, which implies
$\int_{\mathbb{R}^N} |v_n(x)|^{p+2}dx\rightarrow 2\|\nabla Q\|_{L^2}^2$.

When $0 < p_2 <\frac{4}{N}$, note that
the following Gagliardo-Nirenberg inequality
\begin{equation}\label{gn1}
\int_{\mathbb{R}^N} |u(x)|^{p_2+2}dx\leq C\|u\|_{L^2}^{\frac{2(p_2+2)-Np_2}{2}}\|\nabla u\|_{L^2}^{\frac{Np_2}{2}}.
\end{equation}
By the similar argument as \eqref{xx45}, it follows
 that $H(v_n)\rightarrow 0$ as $n\rightarrow\infty$.

When $\frac{4}{N}\leq p_2 <p_1<\frac{4}{N-2}$ and $\frac{Np_1}{2}<p_2+2$, we can obtain the following Gagliardo-Nirenberg inequality
\begin{equation}\label{gn2}
\int_{\mathbb{R}^N} |u(x)|^{p_2+2}dx\leq C\|u\|_{L^{\frac{Np_1}{2}}}^{(p_2+2)(1-\theta_1)}\| \nabla u\|_{L^{2}}^{\theta_1(p_2+2)},
\end{equation}
where $\theta_1=\frac{4p_2+8-2Np_1}{(p_2+2)(2p_1+4-Np_1)}$.
Thus, we deduce from $\theta_1(p_2+2)<2$ that $H(v_n)\rightarrow 0$ as $n\rightarrow\infty$.

When $\frac{4}{N}\leq p_2 <p_1<\frac{4}{N-2}$ and
$\frac{Np_1}{2}>p_2+2$, we have
\begin{equation}\label{gn3}
\int_{\mathbb{R}^N} |u(x)|^{p_2+2}dx\leq C\| u\|_{L^{2}}^{(1-\theta_2)(p_2+2)}\|u\|_{L^{\frac{Np_1}{2}}}^{\theta_2(p_2+2)}\leq C\| u\|_{L^{2}}^{(1-\theta_2)(p_2+2)}\|u\|_{H^{s_c}}^{\theta_2(p_2+2)},
\end{equation}
where $\theta_2=\frac{Np_1p_2}{(p_2+2)(Np_1-4)}$.
This inequality and the assumption \eqref{a} imply that $H(v_n)\rightarrow 0$ as $n\rightarrow\infty$.

When $\frac{4}{N}\leq p_2<p_1<\frac{4}{N-2}$ and $\frac{Np_1}{2}=p_2+2$, it follows from the Sobolev embedding that
\begin{equation}\label{gn4}
\int_{\mathbb{R}^N} |u(x)|^{p_2+2}dx=\|u\|_{L^{\frac{Np_1}{2}}}^{p_2+2}\leq C\|u\|_{H^{s_c}}^{p_2+2},
\end{equation}
which, together with the assumption \eqref{a}, implies that $H(v_n)\rightarrow 0$ as $n\rightarrow\infty$.

Set $m=2\|\nabla Q\|_{L^2}^2$ and $M=\|\nabla Q\|_{L^2}^2$. Then it follows from Lemma 4.1 that there exist $V\in \dot{H}^{s_c}\cap \dot{H}^1$ and $\{x_n\}_{n=1}^\infty \subset \mathbb{R}^N$ such that, up to a subsequence,
\[
v_n(\cdot +x_n)=\rho_nu(t_n,\rho_n\cdot + x_n)\rightharpoonup V~~weakly~in~\dot{H}^{s_c}\cap \dot{H}^1
\]
with
\begin{equation}\label{14.1}
\|V\|^2_{\dot{H}^{s_c}}\geq \frac{m}{2M}\|Q\|_{\dot{H}^{s_c}}^2.
\end{equation}
By the definition of $\dot{H}^{s_c}$, we have
\[
(-\Delta)^{\frac{s_c}{2}}\rho_nu(t_n,\rho_n\cdot + x_n)\rightharpoonup (-\Delta)^{\frac{s_c}{2}}V~~weakly~in~L^2.
\]
Thus, for any $R>0$,
\[
\int_{|x|\leq R}|(-\Delta)^{\frac{s_c}{2}}V(x)|^2dx\leq \liminf_{n\rightarrow \infty} \int_{|x-x_n|\leq \rho_nR}|(-\Delta)^{\frac{s_c}{2}}u(t_n,x)|^2dx.
\]
In view of the assumption $\lambda(t_n)/\rho_n\rightarrow \infty$, this implies immediately
\begin{equation*}
\int_{|x|\leq R}|(-\Delta)^{\frac{s_c}{2}}V|^2dx\leq \liminf_{n\rightarrow \infty} \sup_{y\in \mathbb{R}^N}\int_{|x-y|\leq\lambda (t_n)}|(-\Delta)^{\frac{s_c}{2}}u(t_n,x)|^2dx.
\end{equation*}
Since the sequence $\{t_n\}_{n=1}^\infty$ is arbitrary, we obtain
\begin{equation*}
\int_{|x|\leq R}|(-\Delta)^{\frac{s_c}{2}}V|^2dx\leq \liminf_{n\rightarrow \infty} \sup_{y\in \mathbb{R}^N}\int_{|x-y|\leq \lambda (t)}|(-\Delta)^{\frac{s_c}{2}}u(t,x)|^2dx.
\end{equation*}
Observe that for every $t\in [0,T)$, the function $y\mapsto \int_{|x-y|\leq \lambda (t)}|(-\Delta)^{\frac{s_c}{2}}u(t,x)|^2dx$ is continuous and goes to zero at infinity. Thus, there exists $x(t)\in \mathbb{R}^N$ such that
\begin{equation*}
\int_{|x-x(t)|\leq \lambda (t)}|(-\Delta)^{\frac{s_c}{2}}u(t,x)|^2dx= \sup_{y\in \mathbb{R}^N}\int_{|x-y|\leq \lambda (t)}|(-\Delta)^{\frac{s_c}{2}}u(t,x)|^2dx.
\end{equation*}
This and \eqref{14.1} yield \eqref{c}.
The proof of \eqref{d} is similar, so we omit it. This completes the proof.
\end{proof}

\end{document}